\documentclass[a4paper,11pt,twoside,openright]{article} 

\usepackage{amssymb}
\usepackage{amsthm}
\usepackage{amsmath}
\usepackage{array} 
\usepackage{graphicx}
\usepackage[a4paper,portrait]{geometry}
\newtheorem{theo}{Theorem}[section]  

\newtheorem{rem}[theo]{Remark}
\newtheorem{lem}[theo]{Lemma}
\newtheorem{prop}[theo]{Proposition}
\newcommand{\R}{\mathbb{R}}

\newcommand{\N}{\mathbb{N}}

\newcommand{\E}{\mathbb{E}}

\title{Unbiased risk estimation method for covariance estimation}
\author{H\'el\`ene Lescornel
\footnote{Institut de Math\'ematiques de Toulouse
 UMR 5219
31062 Toulouse, Cedex 9, France. Email: lescornel@math.univ-toulouse.fr}, Jean-Michel Loubes\footnote{loubes@math.univ-toulouse.fr}, Claudie Chabriac\footnote{chabriac@univ-tlse2.fr}
}

\date{}
\begin{document}
\maketitle

\begin{abstract}
We consider a model selection estimator of the covariance of a random process. Using  the Unbiased Risk Estimation (URE) method, we build an estimator of the risk which allows to select an estimator in a collection of model. Then, we present an oracle inequality  which ensures that the risk of the selected estimator is close to the risk of the oracle. Simulations show the efficiency of this methodology.
\end{abstract}
\section{Introduction}
Estimating the covariance function  of stochastic
processes is a fundamental issue in statistics with many applications, ranging from  geostatistics,
financial series or epidemiology for instance (we refer to \cite{MR1697409}, 
\cite{MR0456314} or \cite{MR1239641} for general references). While parametric methods have been extensively studied in the
statistical literature (see \cite{MR1239641} for a review),
nonparametric procedures have only recently received  attention, see for instance \cite{MR2403106,MR2684389,BIGOT:2010:HAL-00440424:4,Bigot:arXiv1010.1601} and references therein. One of the main difficulty in this framework is to impose that the estimator is also a covariance function, preventing the direct use of usual nonparametric statistical methods.\\
\indent  In this paper, we propose to construct a non parametric estimator of the covariance function of a stochastic process by using a model selection procedure based on the Unbiased Risk Estimation (U.R.E.) method. We work under general assumptions on the process, that is, we do not assume Gaussianity nor stationarity of the observations.\vskip .1in

Consider a stochastic process $\left( X\left( t \right) \right)_{t \in T}$ taking its values in $\R$ and indexed by $T \subset\R^d$, $d \in \N$. We assume that $\E \left[ X\left( t \right) \right]=0$ $\forall t \in T$ and we aim at estimating its covariance function $ \sigma \left( s, t \right)= \E \left[X \left( s\right) X\left( t \right) \right] < \infty$ for all $ t, s \in T$. We assume we observe  $X_i\left(t_j \right)$ where  $i\in \left\lbrace 1 \dots n \right\rbrace$ and $t \in \left\lbrace 1 \dots n\right\rbrace$. Note that the observation points $t_j$ are fixed and that the $X_i$'s are independent copies of the process $X$. Set set $x_i=\left(X_i\left(t_1\right),\dots , X_i\left( t_p\right) \right) \forall i \in \left\lbrace 1 \dots n \right\rbrace$ and denote by $\Sigma$ the covariance matrix of these vectors. \vskip .1in
Following the methodology presented in \cite{MR2684389}, we approximate the process $X$ by its projection onto some finite dimensional model. For this, consider a countable set of functions $\left(g_{\lambda}\right)_{\lambda \in \Lambda }$ which may be for instance a basis of $L^2 \left( T\right)$ and choose a collection of models $\mathcal{M} \subset \mathcal{P}\left( \Lambda \right)$. For  $m \subset \mathcal{M}$,  a finite number of indices, the process can be approximated by 
$$X\left( t\right)   \approx \sum_{\lambda \in m} a_{ \lambda} g_{\lambda}\left( t\right). $$ Such an approximation leads to an estimator of $\Sigma$ depending on the collection of functions $m$, denoted by $\hat{\Sigma}_m$. Our objective is to select in a data driven way, the best model, i.e the one close to an oracle $m_0$ defined as a minimizer of the quadratic risk, namely $$m_0 \in \underset{m\in \mathcal{M}}{{\rm arg}\min}R\left(m\right) = \underset{m\in \mathcal{M}}{{\rm arg}\min} \E\left[ \left\Vert \Sigma - \hat{\Sigma}_m\right\Vert^ 2 \right].$$
A model selection procedure will be performed using  the U.R.E. method, which has been introduced in \cite{STE} and fully described in \cite{TSYB}. The idea is to find an estimator $\hat{R}\left(m \right)$ of the risk which is unbiased, and to select  $\hat{m}$ by minimizing this estimator. Hence, if $\hat{R}$ is close to its expectation, $\hat{\Sigma}_{\hat{m}}$ will be an estimator with a small risk, nearly as the best quantity $\hat{\Sigma}_{m_0}$. \vskip .1in

In this work, following the U.R.E. method, we build an estimator of the risk which allows to select an estimator of the covariance function. Then, we present an oracle inequality for the covariance estimator which ensures that the risk of the selected estimator is not too large with respect to the risk of the oracle. \vskip .1in

The paper is organized as follows. In Section~\ref{sec:stat frame} we present the statistical framework and  recall some useful algebraic tools for matrices. The following section, Section~\ref{sec:model selec} is devoted to the approximation of the process and the construction of the covariance estimator. Section~\ref{sec:ure} is devoted to the U.R.E. method, and provides an oracle inequality. Some numerical experiments are exposed in Section~\ref{sec:num}, while the proofs are postponed to the Appendix.

\section{The statistical framework}
\label{sec:stat frame}
Recall that we consider an $\R$-valued stochastic process, $X=\left(X\left( t\right)\right)_{t \in T}$, where $T$ is some subset of $\R^d$, $d \in \N$. We assume that $X$ has finite moments up to order 4 and zero mean. Our aim is to study the covariance function of $X$ denoted by $\sigma\left(s,t\right)=\E \left[ X\left( s\right) X\left( t \right) \right]$.

Let $X_1,... X_n$ be independent copies of the process $X$, and  assume that we observe these copies at some determinist points $t_1,...,t_p$ in $T$. We set $x_i=\left(X_i\left(t_1\right), \dots , X_i\left( t_p \right) \right)\top$, and denote the empirical covariance of the data by
$$ S = \frac{1}{n} \sum_{i=1}^n x_i x_i ^\top$$ with expectation $\Sigma = \left(\sigma\left( t_j,t_k\right)\right)_{1\leqslant j, k \leqslant p}$.\vskip .1in

Hence, the observation model can be written, in a matrix regression framework, as
\begin{equation}
\label{eq:modmatr}
x_ix_i^\top=\Sigma + U_i \quad \in \mathbb{R}^{p\times p} \quad, 1\leqslant i \leqslant n
\end{equation}
Where $U_i$ are i.i.d. error matrices with $\mathbb{E}\left[   U_i\right]  =0$.\vskip .1in
We now recall some notations related to the study of matrices, which will be used in the following.\\
Denote by $\mathcal{S}_t$ the subset composed of symmetric matrix in $\R^{t\times t}$. 

For any matrix $A=\left( a_{ij}\right)   \in \mathbb{R}^{s \times t}$,  $ \left\Vert A \right\Vert^2 = tr\left( AA^\top\right)  $ is the Frobenius norm of the matrix which is associated to the inner scalar product $\left\langle  A,B\right\rangle  =tr\left( AB^\top\right)  $. 

$A^- \in \mathbb{R}^{t \times s}$ is a reflexive generalized  inverse of $A$, that is, some matrix such as $A^-AA^-=A$ and$AA^-A= A^-$.

In the following, we will consider matrix
data as a natural extension of the vectorial data, with different
correlation structure. For this, we introduce a natural linear
transformation, which converts any matrix into a column vector. The
vectorization of a $k\times n$ matrix ${A}=(a_{ij})_{1\leq i\leq
k,1\leq j\leq n}$ is the $kn\times 1$ column vector denoted by $vec\left( 
{A}\right) $, obtained by stacking the columns of the matrix ${A}
$ on top of one another. That is $%
vec(A)=[a_{11},...,a_{k1},a_{12},...,a_{k2},...,a_{1n},...,a_{kn}]^{\top }$.%

If ${A=}(a_{ij})_{1\leq i\leq k,1\leq j\leq n}$ is a $k\times n$
matrix and $\mathbf{B=}(b_{ij})_{1\leq i\leq p,1\leq j\leq q}$ is a $p\times
q$ matrix, then the Kronecker product of the two matrices, denoted by $%
{A}\otimes {B}$, is the $kp\times nq$ block matrix%
\begin{equation*}
{A}\otimes {B=}%
\begin{bmatrix}
a_{11}{B} & . & . & . & a_{1n}{B} \\ 
. & . &  &  & . \\ 
. &  & . &  & . \\ 
. &  &  & . & . \\ 
a_{k1}{B} & . & . & . & a_{kn}{B}%
\end{bmatrix}%
.
\end{equation*}


For $A$, $B$ and $C$  some real matrices, we recall the following properties that will be useful in our settings.  
\begin{prop}
\label{rappel}
\begin{equation}
\label{rap:eq1}
vec\left( ABC\right)  =\left( C^\top \otimes A\right)   vec\left( B\right)  
\end{equation}
\begin{equation}
\label{rap:eq2}
\left\Vert A \right\Vert = \left\Vert vec\left( A\right)   \right\Vert = \left\Vert vec\left( A\right)   \right\Vert_{\ell_2}
\end{equation}
\begin{equation}
\label{rap:eq3}
\left( A \otimes B\right)  \left( C \otimes D\right)= \left(AC \right)\otimes \left(BD\right)
\end{equation}
\begin{equation}
\label{rap:eq4}
\left( A \otimes B\right)  ^\top= A^\top \otimes B^\top
\end{equation}
\end{prop}



These identities can be found in \cite{ALGMATR}.\vskip .1in

Let $m\in \mathcal{M}$, and recall that to the finite set $\left\{ g_{\lambda }\right\} _{\lambda \in m}$ of functions $g_{\lambda
}:T\rightarrow \mathbb{R}$ we associate the $n\times |m|$ matrix $\mathbf{G}$
with entries $g_{j\lambda }=g_{\lambda }\left( t_{j}\right) $, $j=1,...,n$, $%
\lambda \in m$. Furthermore, for each $t\in T$, we write $\mathbf{G}%
_{t}=\left( g_{\lambda }\left( t\right) ,\lambda \in m\right) ^{\top }$. For 
$k\in \mathbb{N}$, $\mathcal{S}_{k}$ denotes the linear subspace of $\mathbb{%
R}^{k\times k}$ composed of symmetric matrices. For $\mathbf{G\in }\mathbb{R}%
^{n\times |m|}$, $\mathcal{S}\left( \mathbf{G}\right) $ is the linear
subspace of $\mathbb{R}^{n\times n}$ defined by%
\begin{equation*}
\mathcal{S}\left( \mathbf{G}\right) =\left\{ \mathbf{G\Psi G}^{\top }:%
\mathbf{\Psi \in }\mathcal{S}_{m}\right\} \text{.}
\end{equation*}%
This set will be the natural projection space for the corresponding covariance estimator.


\section{Model selection approach}
\label{sec:model selec}
The estimation procedure is a two step procedure. First we consider a functional expansion of the process and  approximate it by its projection onto some finite collection of functions. Then, we construct a rule to pick out the best of these estimators among the collection of estimated,  based on the U.R.E. method.\vskip .1in
In this section, we explain  the construction of a projection based estimator for the covariance of a process and point out its properties.  More details can be  found in \cite{MR2684389}.\vskip .1in

Consider a process $X$ with an expansion on a set of functions $\left( g_\lambda\right) _{\lambda \in \Lambda}$ of the following form  $$X\left( t\right)  =\sum_{\lambda \in \Lambda} a_{ \lambda} g_{\lambda}\left( t\right)  $$
where $\Lambda$ is a countable set, and $\left( a_\lambda\right)  _{\lambda \in \Lambda}$ are random coefficients in $\R$ of the process $X$. \\
\indent This situation occurs in large number of cases. If we assume that the process takes its values in $L^2\left( T \right)$ or an Hilbert space, a natural choice of the functions is given by the corresponding Hilbert basis $\left( g_\lambda\right)  _{\lambda \in \Lambda}$  of $L^2\left( T \right)$.  Alternatively, the Karhunen-Loeve expansion of the covariance provides a natural basis. However, since it relies on the nature of the process $X$, this expansion is usually  unknown or require additional information on the process. We refer to \cite{ADL} for more references on this expansion. Under other kind of regularity assumptions on the process, for instance assuming  that the paths of the process  belong to some  RKHS,  other expansions can be considered as  in~\cite{CAI} for instance.

Now consider the projection of the process onto a finite number of functions. For this, let  $m$ be a finite subset of $\Lambda$ and  consider the  corresponding approximation of the process in the following form 
\begin{equation}
\tilde{X}\left( t\right)   = \sum_{\lambda \in m} a_{ \lambda} g_{\lambda}\left( t\right)  
\end{equation}

We note $G_m \in \mathbb{R}^{p \times \left\vert m \right\vert}$ where $\left( G_m\right)  _{j\lambda}= g_\lambda \left( t_j\right)  $ and $a_m$ the random vector of $\mathbb{R}^{|m|}$ with coefficients $\left( a_\lambda\right)  _{\lambda \in m}$.

Hence, we obtain that  
$$\tilde{x}=\left( \tilde{X}\left( t_1\right)  ,..,\tilde{X}\left( t_p\right)  \right)  ^\top = G_{m}a_m$$
and 
$$\tilde{x}\tilde{x}^\top = G_{m} a_m a_m^\top G_m^\top.$$

Thus, approximating the process $X$ by $\tilde{X}$ its projection onto the model $m$ implies  approximating the covariance matrix $\Sigma$ by $ G_{m} \Psi G_m^\top \quad \Psi \in \mathbb{R}^{|m|\times |m|}$
where $\Psi= \mathbb{E}\left[    a_ma_m^\top\right]$ is some symmetric matrix. With previous definitions, that amounts to saying that we want to choose an estimator in the subset $\mathcal{S}\left( G_m \right)$ for some subset $m$ of $\Lambda$.\vskip .1in


Assume that the subset $m$ is fixed. The best approximation of $\Sigma$ in  $\mathcal{S}\left( G_m \right)$  for the Frobenius norm is its projection denoted by $\Sigma_m$. But $\Sigma$ is unknown, hence we can not determinate this quantity. A natural idea is to study the projection of $S$ on $\mathcal{S}\left( G_m \right)$. We denote this quantity by $\hat{\Sigma}_m$.

Proposition 3.1 in \cite{MR2684389} gives an explicit form for these projections. We recall it for sake of completeness.

\begin{prop}
\label{proproj}
Let $A$ in $\mathbb{R}^{p \times p}$ and $G\in \mathbb{R}^{p \times |m|}$.
The infimum $$\inf \left\{\left\Vert A - \Gamma \right\Vert ; \Gamma \in \mathcal{S}\left( G\right)   \right\}$$
is achieved at $$\hat{\Gamma} = G \left( G^\top G\right)  ^- G^\top \left( \frac{A+A^\top}{2}\right) G  \left( G^\top G\right)  ^- G^\top $$
In particular, if $A\in \mathcal{S}_p$, the projection of $A$ on $\mathcal{S}\left( G\right)  $ is $ \Pi A \Pi$ with the projection matrix $\Pi=G \left( G^\top G\right)  ^- G^\top \in \mathbb{R}^{p\times p}$.

It amounts to saying that $\inf \left\{\left\Vert A - G \Psi G^\top \right\Vert ; \Psi \in \mathcal{S}_{|m|} \right\}$ is reached at $$\hat{\Psi} = \left( G^\top G\right)  ^- G^\top \left( \frac{A+A^\top}{2}\right)   G \left( G^\top G\right)  ^-.$$

\end{prop}

\begin{rem}
\label{rq:ouf}
Thanks to the properties of the reflexive generalized  inverse given in \cite{ALGRAO}, the projection of a non-negative definite matrix $A\in \mathcal{S}_p$ on $\mathcal{S}\left( G\right) $ will be also a non-negative definite matrix. Moreover, the matrix $\Pi$ does not depend on the choice of the generalized inverse.
\end{rem}

Thanks to this result, the projection of $\Sigma$ on $\mathcal{S}\left(G_m\right)$ can be characterized as
 \begin{equation}
 \Sigma_m= \Pi_m \Sigma \Pi_m
 \end{equation}
 
 and the same for $S$ (that is, our candidate for estimating $\Sigma$)
\begin{equation}
\hat{\Sigma}_m= \Pi_m S \Pi_m
\end{equation}

where $\Pi_m= G_m \left( G_m^\top G_m\right)  ^- G_m^\top$.

Hence, the estimator $\hat{\Sigma}_m$ is a covariance matrix. Now, our aim is to choose the best subset $m$ among a collection of candidates. 

\section{Model selection with the U.R.E. method}
\label{sec:ure}
Let $\mathcal{M}$ be a finite collection of models $m$. In this section, we focus on picking the best model among this collection by following the U.R.E. method. Since the law of $\left\Vert \Sigma - \hat{\Sigma}_m \right\Vert$ is unknown, we thus  aim at  finding an estimator of its expectation.

We consider that the best subset $m$ is  $m_0$ defined by
$$m_0 \in \underset{m\in \mathcal{M}}{argmin} \E \left[\left\Vert \Sigma - \hat{\Sigma}_m \right\Vert^2\right]$$ 
Then the oracle is defined as the best estimate knowing all the information, namely $\hat{\Sigma}_{m_0}$.

Set $R(m)=\E \left[\left\Vert \Sigma - \hat{\Sigma}_m \right\Vert^2\right]$.
First, we compute this quantity.
\begin{prop}
\label{eq:risk}
\begin{equation}
\mathbb{E}\left[   \left\Vert \Sigma -\hat{\Sigma}_m \right\Vert^ 2\right]  = \left\Vert \Sigma - \Pi_m \Sigma \Pi_m\right\Vert^2 + \frac{tr\left(  \left( \Pi_m \otimes \Pi_m\right)   \Phi\right)}{n}
\end{equation}
where $\Phi=Var\left( vec\left( xx^\top\right)  \right)  $. 
\end{prop}

Here we can note the similarity with the usual risk for standard estimation models.
For instance, assume that we observe a Gaussian model with observations a vector $Y\in \mathbb{R}^n$ such as 
$$Y=\theta + \epsilon\xi \quad  \xi \sim \mathcal{N}\left( 0,I_n\right) $$
where $\epsilon \in \R$ and $\theta \in \mathbb{R}^n$ is the unknown quantity to estimate, using the projection $\hat{\theta}_m$ of the vector $Y$ onto some subspace $S_m$. If the subspace dimension is denoted by $D_m$, the risk of a such estimator is given by
$$\E\left[ \left\Vert \theta- \hat{\theta}_m \right\Vert^ 2 \right]   =\left\Vert \theta_m- \theta \right\Vert^2 + \epsilon^2D_m.$$
We thus recognize the same kind of decomposition with a bias term and with $\frac{tr\left(  \left( \Pi_m \otimes \Pi_m\right)   \Phi\right)}{n}$ playing the role of the variance term $D_m/n$ with $\epsilon = 1/\sqrt{n}$.  Hence it is natural to extend the Unbiased Risk Estimation procedure of previous Gaussian model to the matrix model obtained by the vectorization of Model~\eqref{eq:modmatr}.

Now, we present an estimator of the risk.
We assume $n \geqslant 3$, and we set :

$$\hat{\gamma}_m^2= \frac{1}{n-1}\sum_{i=1}^n \left\Vert\Pi_m x_ix_i^\top \Pi_m - \hat{\Sigma}_m \right\Vert^2$$

\begin{prop}
\label{ure}
$\left\Vert S- \hat{\Sigma}_m\right\Vert^ 2 + 2 \frac{\hat{\gamma}_m^2}{n}+ C$ is an unbiased  estimator of the risk,  where $C$ does not depend on $m$.
More precisely : 
$$\mathbb{E}\left[    \left\Vert S- \hat{\Sigma}_m\right\Vert^ 2 + 2 \frac{\hat{\gamma}_m^2}{n}\right]  =\mathbb{E}\left[   \left\Vert \Sigma -\hat{\Sigma}_m \right\Vert^ 2\right]   + \frac{tr\left( \Phi\right)  }{n}$$
\end{prop}

Note that  the constant $\frac{tr\left( \Phi\right)  }{n}$ is unknown but does not depend on $m$. So in the URE procedure, minimizing $\left\Vert S- \hat{\Sigma}_m\right\Vert^ 2 + 2 \frac{\hat{\gamma}_m^2}{n}$ with respect to $m$ is equivalent to minimizing $\left\Vert S- \hat{\Sigma}_m\right\Vert^ 2 + 2 \frac{\hat{\gamma}_m^2}{n}+ C$ which is unbiased. 

Then we can define the estimator $\hat{\Sigma}$ of $\Sigma$ by

$$\hat{\Sigma}=\Pi_{\hat{m}} S \Pi_{\hat{m}}= \Hat{\Sigma}_{\hat{m}}$$
$$\text{with }\hat{m}\in \underset{ m\in \mathcal{M}}{argmin} \left(\left\Vert S- \hat{\Sigma}_m\right\Vert^ 2 + 2 \frac{\hat{\gamma}_m^2}{n}\right) $$

The next theorem establishes an oracle inequality for this  estimator. %
\begin{theo}
\label{oracle}
For all $A>0$, we have :
$$\mathbb{E}\left[   \left\Vert\tilde{\Sigma}- \Sigma \right\Vert^2\right]  \leqslant \left( 1+A^{-1}\right)  \inf _{\substack{m\in \mathcal{M}}}\mathbb{E}\left[   \left\Vert \Sigma -\hat{\Sigma}_{m} \right\Vert^ 2\right]   + \frac{tr\left( \Phi\right)  }{n}\left( 4+A\right)  $$
\end{theo}
Hence we have obtained a model selection procedure which enables to recover
the best covariance model among a given collection. This method works
without strong assumptions on the process, in particular stationarity is not
assumed, but at the expend of necessary i.i.d observations of the process at
the same points. \\ We point out that this study requires a large number of replications $n$ with respect to the number of observation points $p$. Actually our method is not designed  to  tackle the problem of  covariance estimation in the high dimensional case $p>>n$. This topic  has received a growing attention over the past years and we refer to \cite{Levina08} and references therein for a survey. \vskip .1in

The proof of these results are using the vectorization of the matrices involved here. That is why we must deal with the matrix $\Phi=var\left( vec\left( x x^\top \right)\right)$. It is postponed to the appendix.



\section{Numerical examples}
\label{sec:num}
In this section we illustrate the behaviour of the covariance estimator $\hat{\Sigma}$ with programs implemented using SCILAB. We aim at knowing if our procedure leads to choose the best model, that is the model minimizing the risk.\\

Recall that $n$ is the number of copies of the process and $p$ is the number of points where we observe these copies.
Here, we consider the case where $T=\left[   0;1\right]  $ and $\Lambda$ is a subset of $\mathbb{N}$. For sake of simplicity, we identify  $m$ and the set ${1,\dots, m}$. Moreover, the points $\left( t_j\right)_{1 \leqslant j \leqslant p}$ are equi-spaced in $\left[   0;1\right]  $.\\

For a given process $X$, we must start by the choice of the functions of its expansion. Their knowledge is needed for the matrix $G_m$. Indeed, $\left( G_m\right)  _{j\lambda}= g_\lambda \left( t_j\right)$.\\

The method is the following: First, we simulate a sample for $p$ and $n$ given. Second, for $m$ between $1$ to some integer $M$, we compute the unbiased risk estimator related to the model $m$. Finally, we pick out a $\hat{m}$ minimizing this estimator and we compute $\hat{\Sigma}$.

For each example, we plot the curve of the risk function and give its minimum $m_0$. We plot also the curve of the function of the risk estimator and give its minimum $\hat{m}$. Finally we compare the true covariance and the estimator.

\subsubsection*{Example 1}
Here we work with the numerical examples of \cite{MR2684389}.
We choose the Fourier basis functions: 
\begin{equation*}
g_\lambda( t)  =\left \{ \begin{array}{l} \frac{1}{\sqrt{p}} \text{ si } \lambda= 1 \\
                                \sqrt{2}  \frac{1}{\sqrt{p}} \cos( 2\pi \frac{\lambda}{2} t ) \text{ si } \lambda \text{ est pair} \\
                                 \sqrt{2}  \frac{1}{\sqrt{p}} \sin( 2\pi \frac{\lambda -1}{2} t ) \text{ si } \lambda \text{ est impair}  
\end{array} \right.
\end{equation*}

And we study the following process :
$$X(t)=\sum_{\lambda=1}^{m^\star} a_\lambda g_\lambda(t)$$
where $a_\lambda$ are independent Gaussian variables with mean zero and variance $V(a_\lambda)$.
Let $D(V)$ the diagonal matrix in $m^\star \times m^\star$ such as $D(V)_{\lambda \lambda}= V(a_\lambda)$. Then we have
$$\Sigma=G_{m^\star} D(V) G_{m^\star}^\top$$

Here are the results for $V(a_\lambda)=1 \quad \forall \quad \lambda$. We choose $m^\star=35=p$, $n=50$, $M=31$. Here it can be shown that the minimum of the risk is achieved at $\frac{n}{2}-1$, so in this setting we have $m_0=24$. Here the minimum of the estimator is the same : $\hat{m}=24$.

\begin{tabular}{c c}
\includegraphics[width=7 cm]{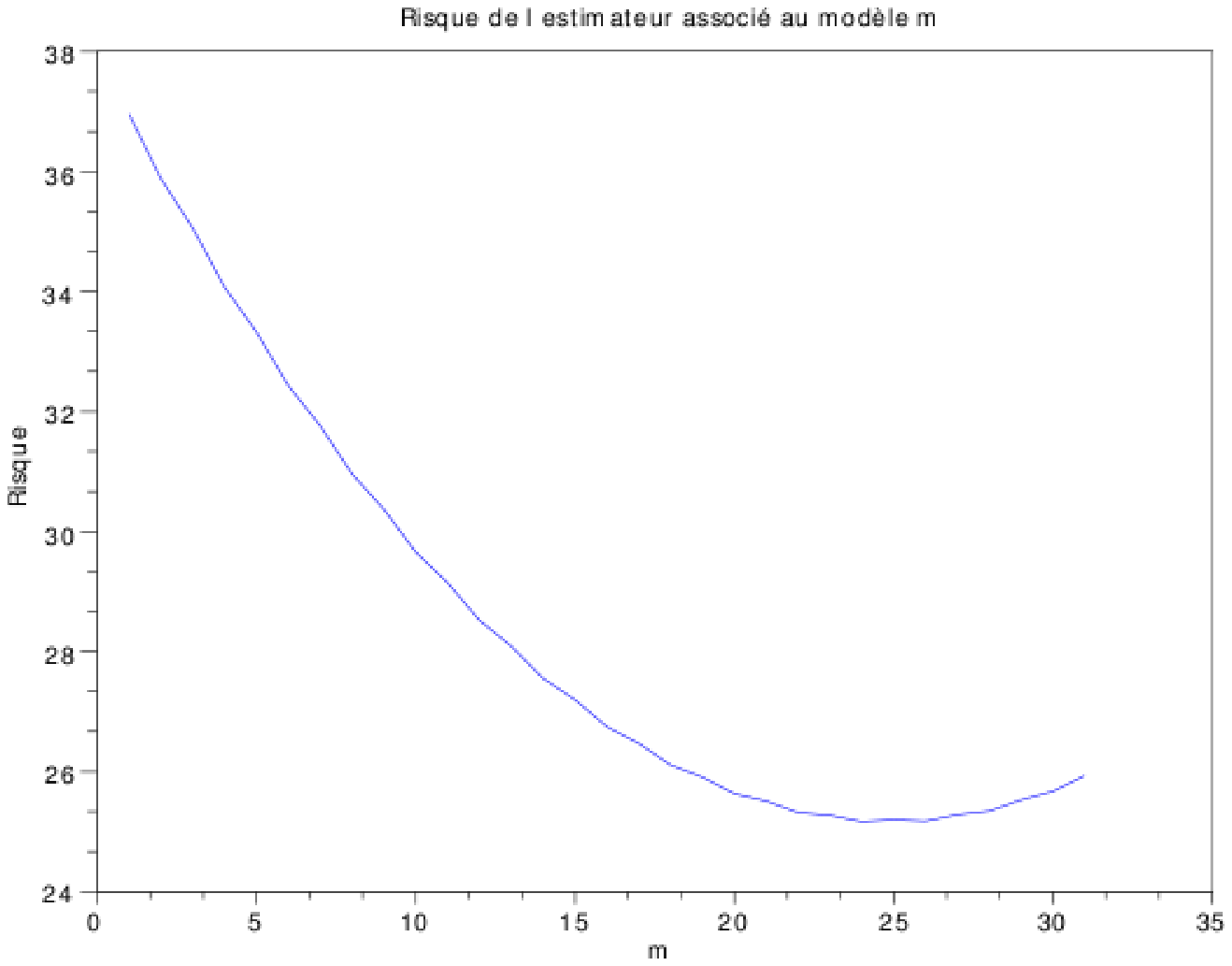} &
\includegraphics[width=7 cm]{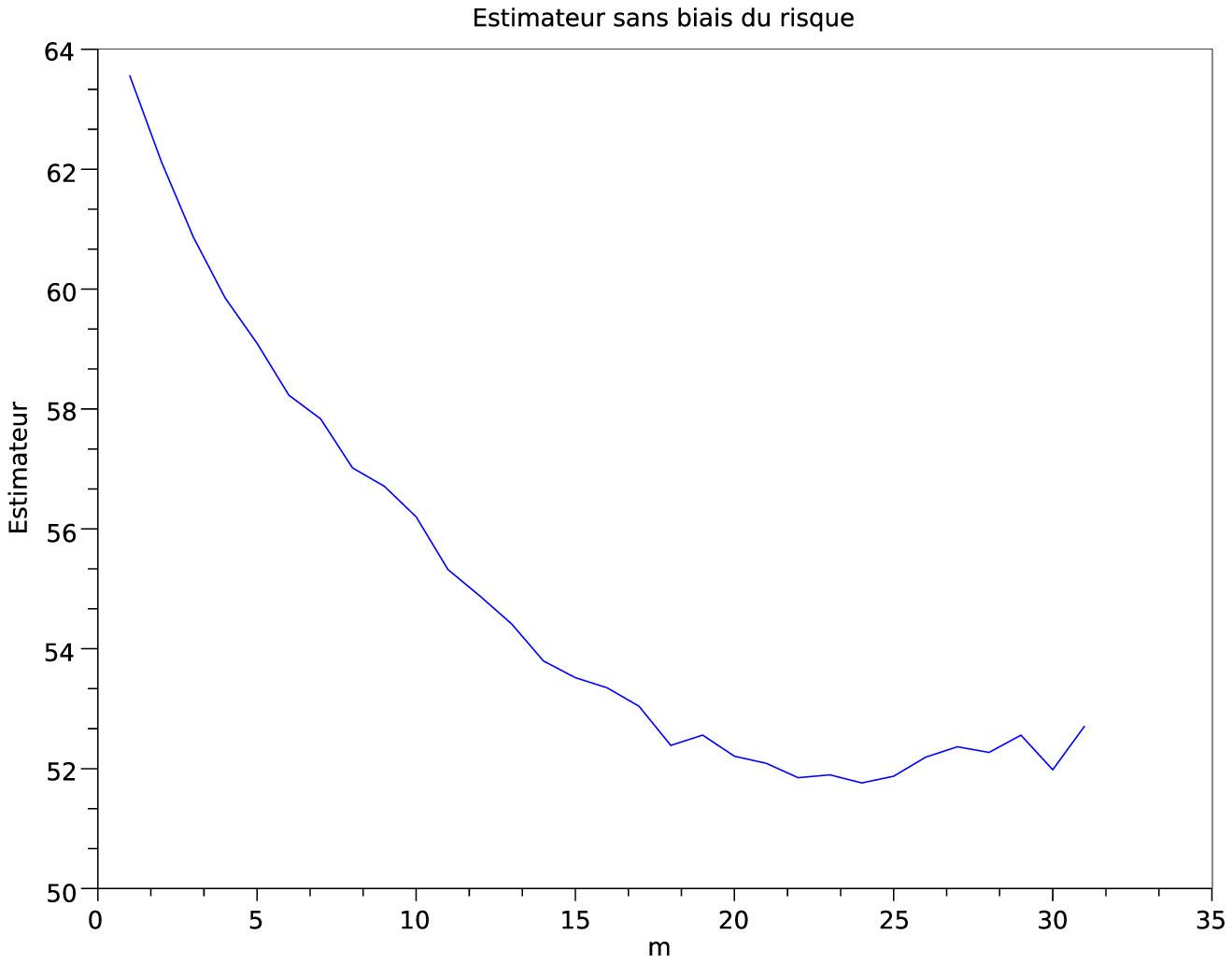} \\
\includegraphics[width=7 cm]{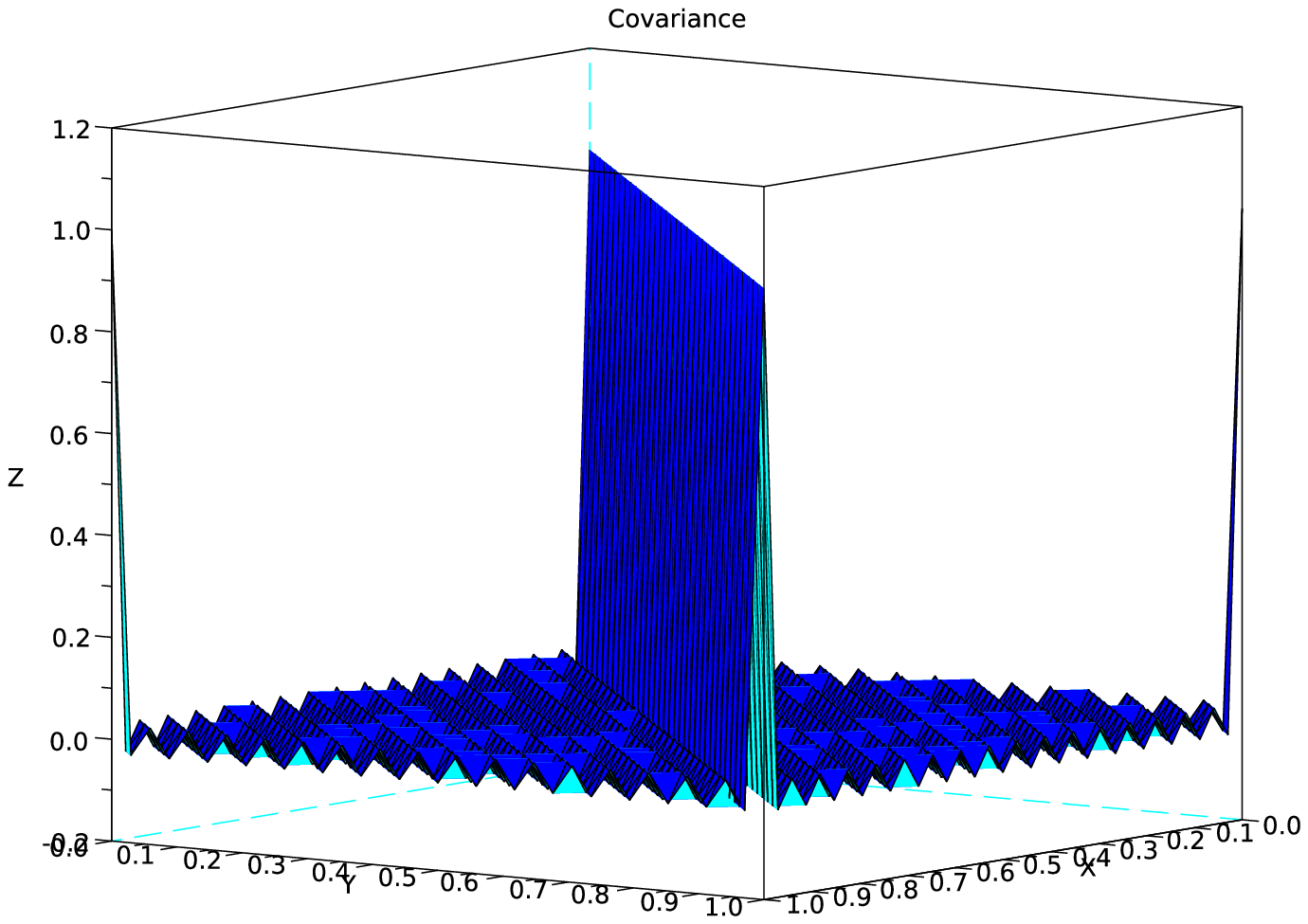}&
\includegraphics[width=7 cm]{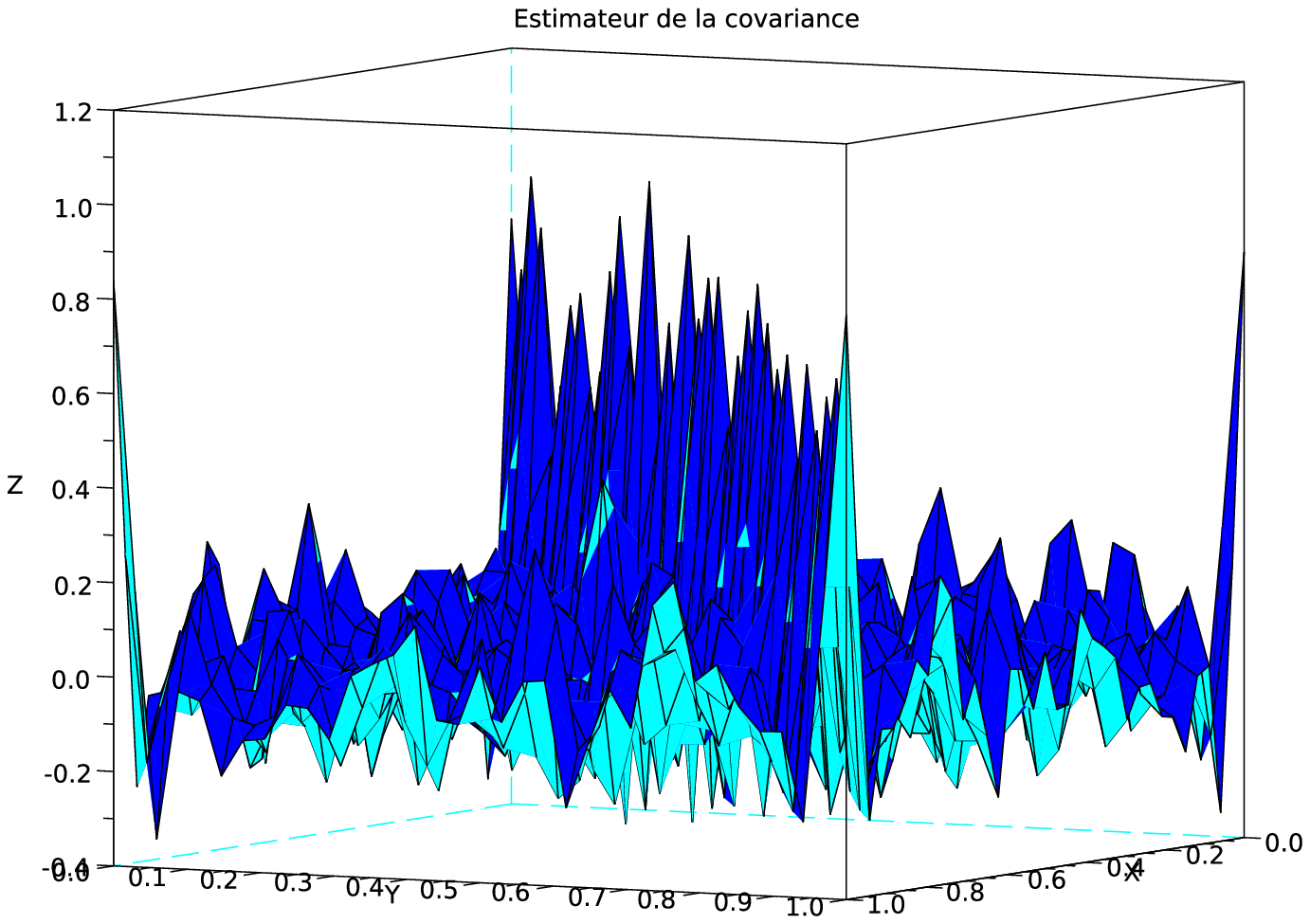}
\end{tabular}



Here are the results for $V(a_\lambda)=0.0475+0.95^{\lambda}\quad \forall \quad \lambda$, and $m^\star=35=p$, $n=60$, $M=34$. Here the figures show that $m_0=\hat{m}=18$.


\subsubsection*{Example 2} 
Now we test our estimator with the process studied in \cite{CAI}.

We consider the functions
$$g_\lambda(t)=\cos(\lambda \pi t)$$
And the process $X$ studied is :
$$X(t)=\sum_{\lambda=1}^{m^\star} a_\lambda \zeta_\lambda g_\lambda(t)$$
where $a_\lambda$ are i.i.d. random variables following the uniform law on $\left[   -\sqrt{3};\sqrt{3}\right]  $ and $\zeta_\lambda=\frac{(-1)^{\lambda + 1}}{\lambda^2}$. 
If $D$ is the diagonal matrix with entries $D_{\lambda \lambda}= \frac{1}{\lambda^{4}}$, as before we have that
$$\Sigma=G_{m^\star} D G_{m^\star}^\top$$
Here we choose $m^ \star =50$, $n=1000$, $p=40$ and $M=20$. We found $m_0=4=\hat{m}$.

\begin{tabular}{c c}
\includegraphics[width=7 cm]{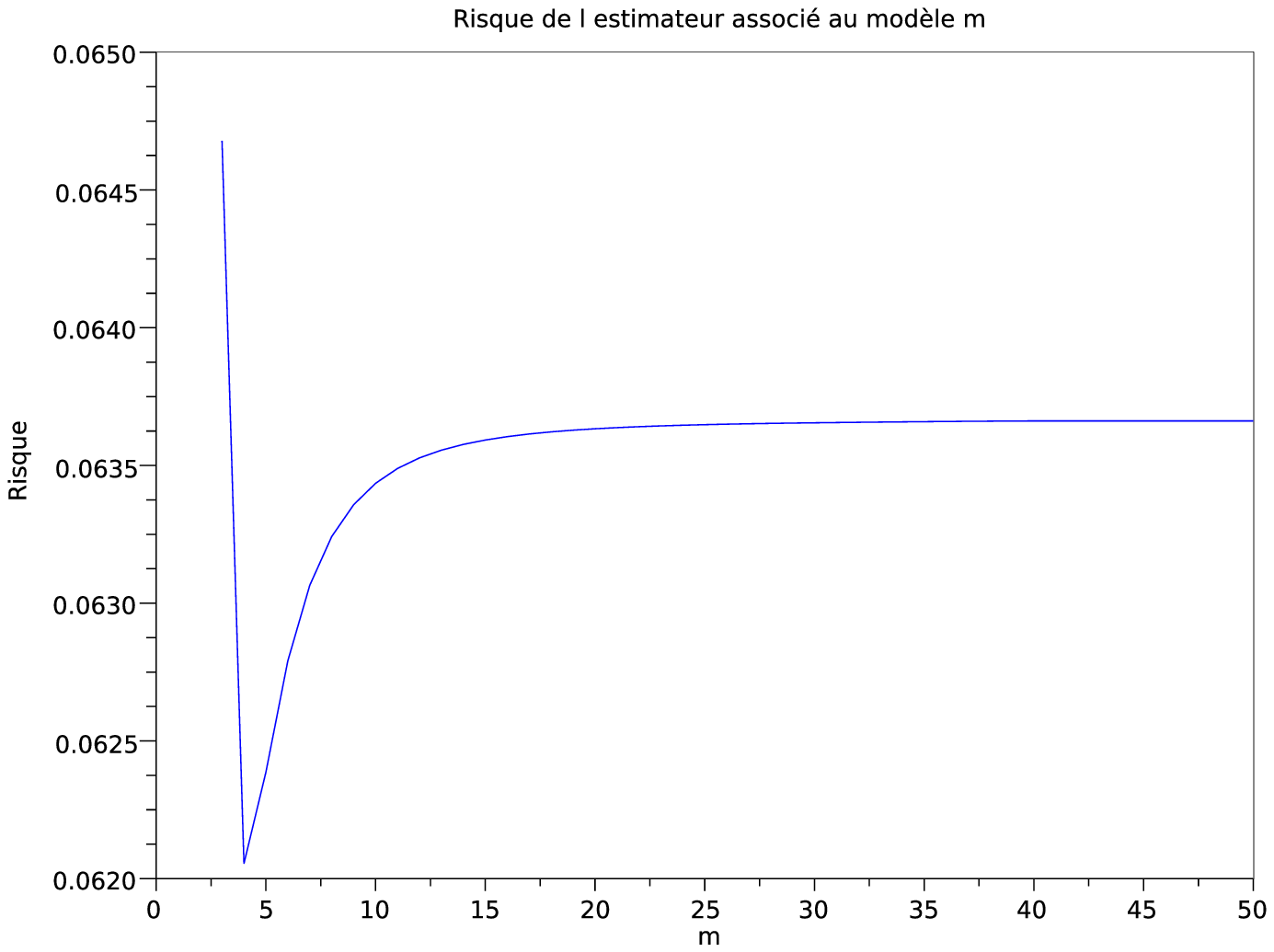} &
\includegraphics[width=7 cm]{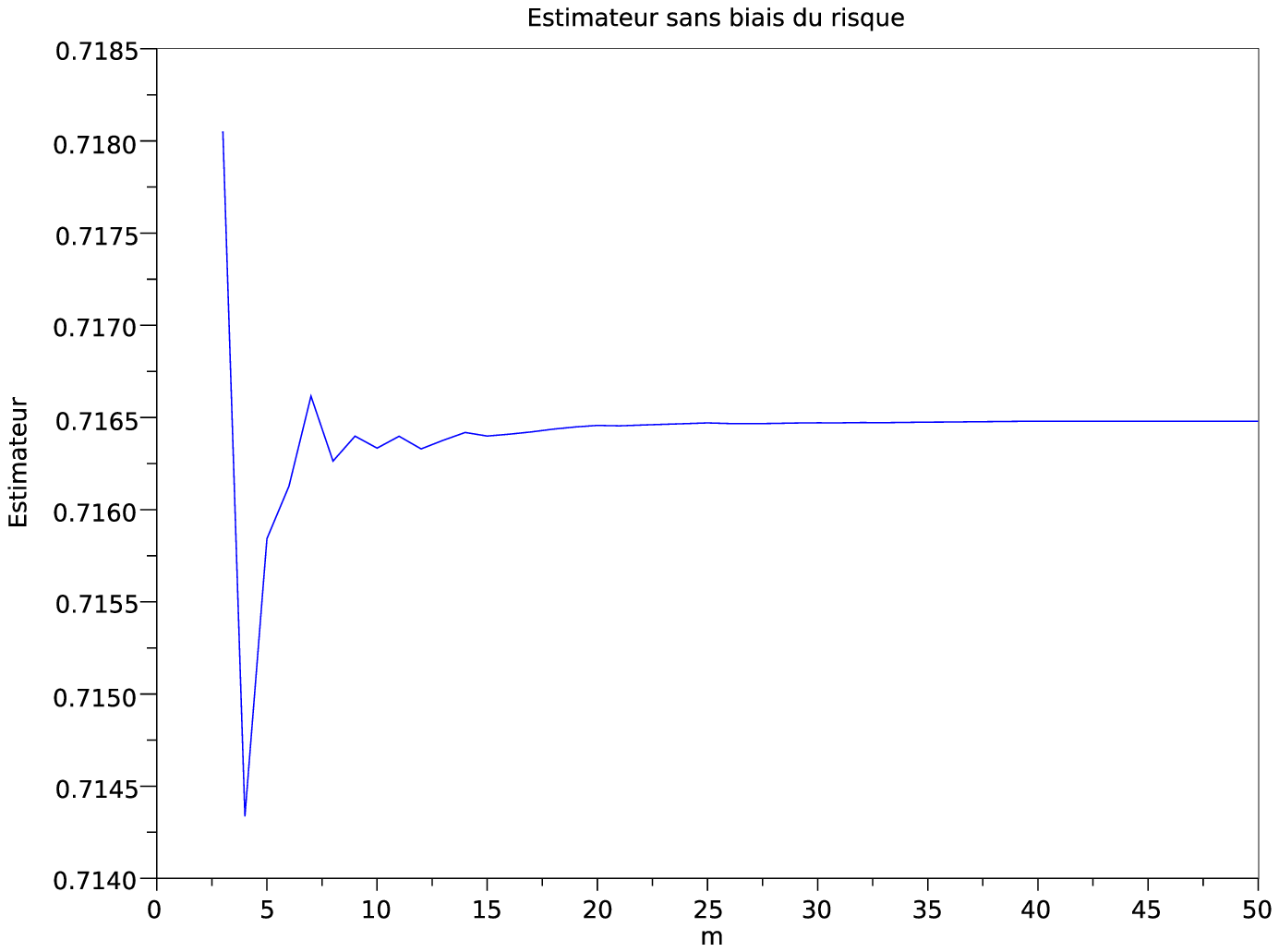} \\
\includegraphics[width=7 cm]{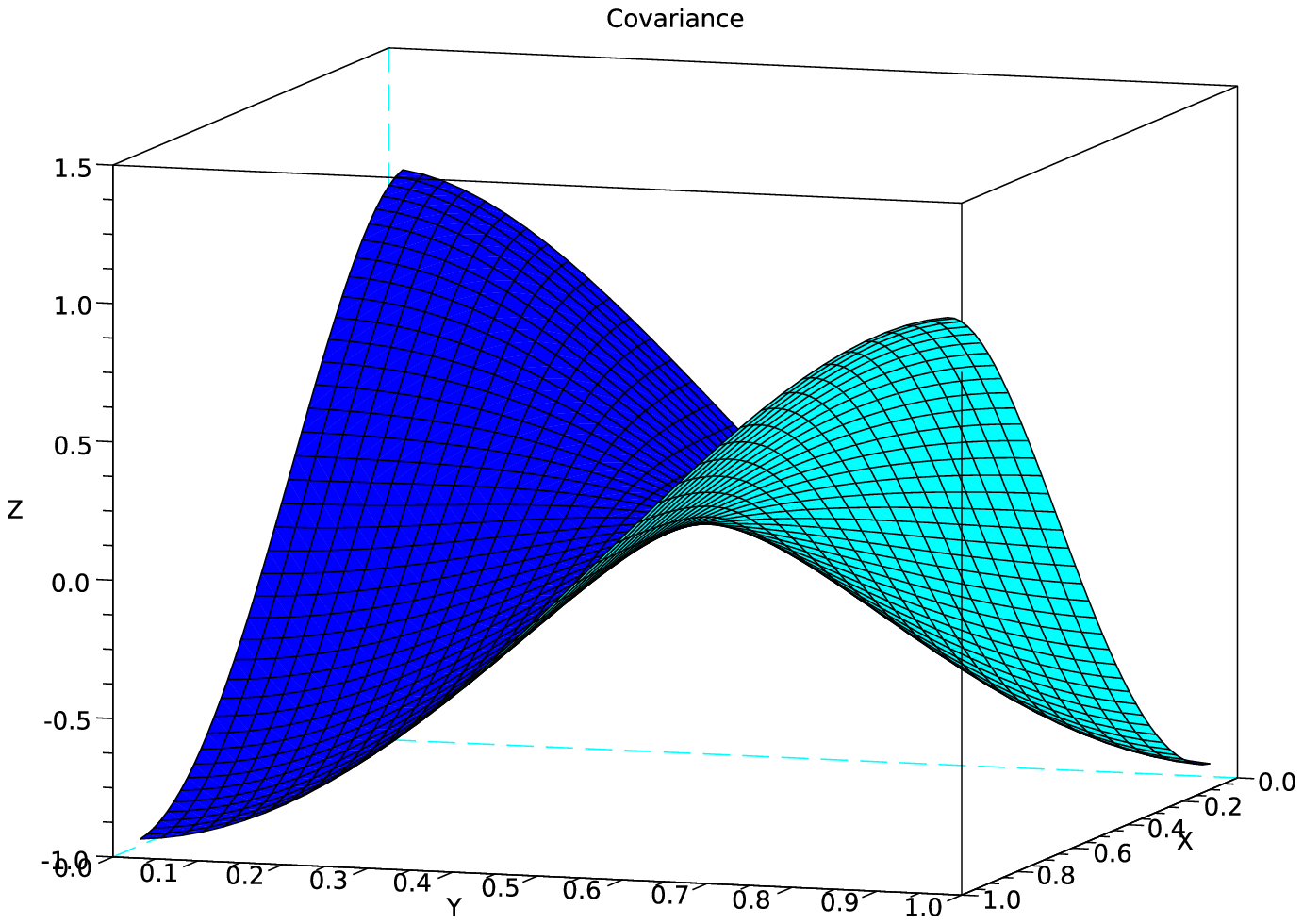}&
\includegraphics[width=7 cm]{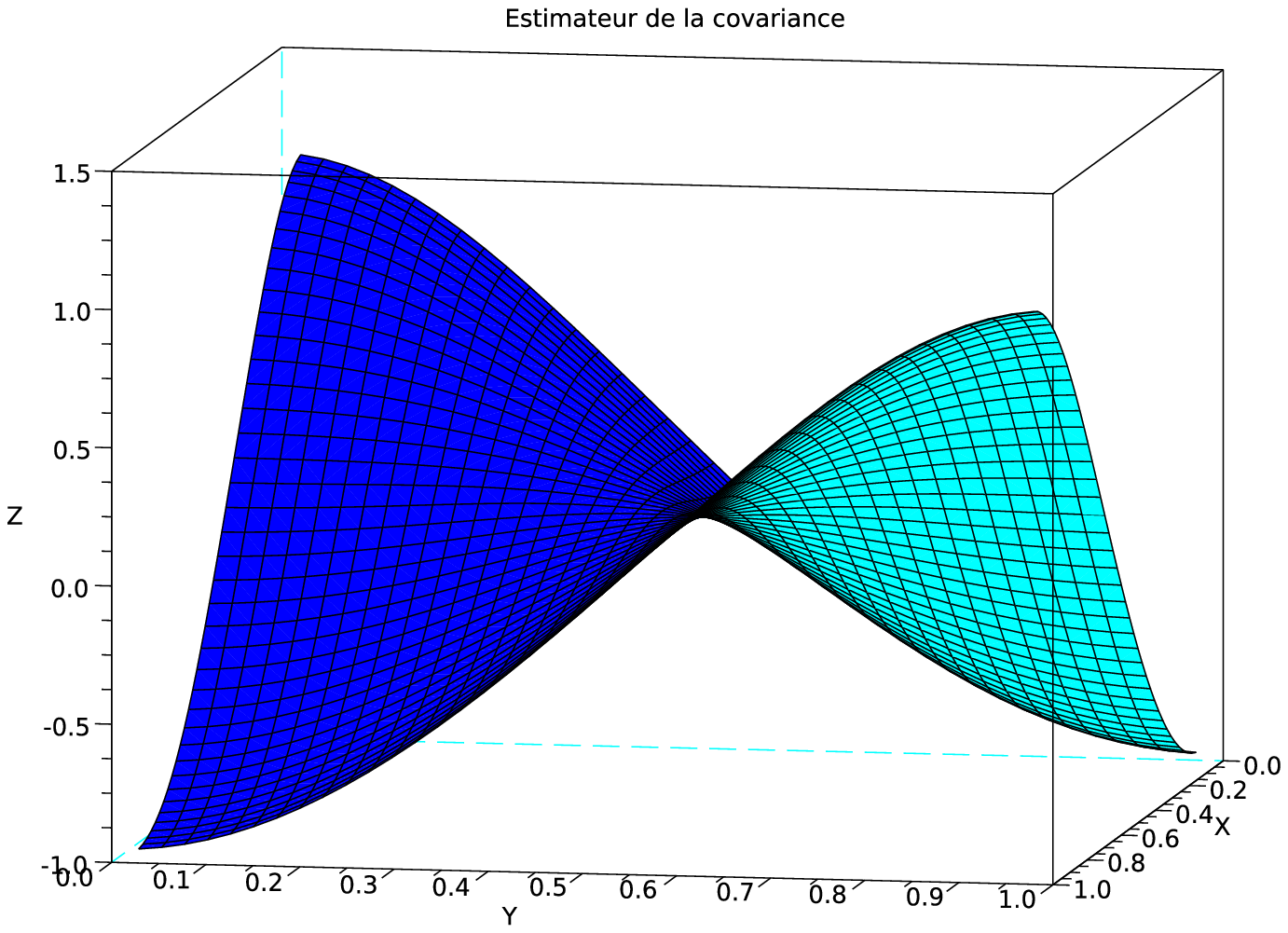}
\end{tabular}

\subsubsection*{Example 3}
Here we consider the case of the brownian bridge with its Karhunen Loeve expansion.
Indeed, this expansion 
$$X(t)= \sum_{\lambda \geqslant 1 } Z_\lambda \sqrt{\nu_\lambda}g_\lambda(t)$$
is computed in  \cite{SHORWELL}, p.213-215 : $\nu_\lambda= \left( \frac{1}{\lambda \pi}\right) ^2$, and $g_\lambda(t)=\sqrt{2}\sin(\lambda \pi t)$.

The covariance function of the brownian bridge is $K(s,t)=s(1-t)$ for $s\leqslant t$.
Simulate the sample is the same as simulate $n$ gaussian vectors of covariance matrix $\Sigma=(K(t_i,t_j))_{1\leqslant i, j \leqslant p}$.
 Here $n=100$, $p=35$ and $M=20$. We found $m_0=5=\hat{m}$.

\begin{tabular}{c c}
\includegraphics[width=7 cm]  {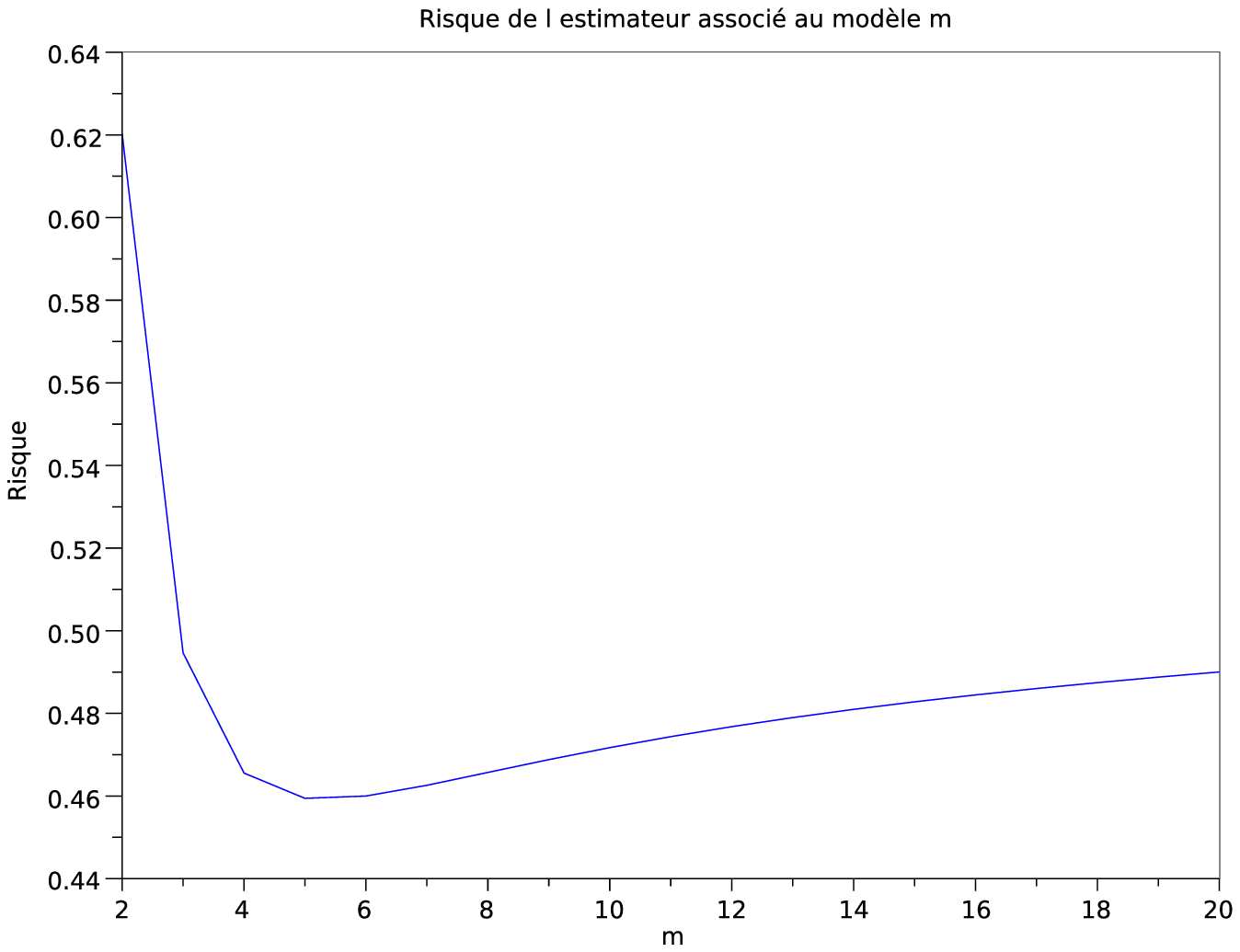} &
\includegraphics[width=7 cm]  {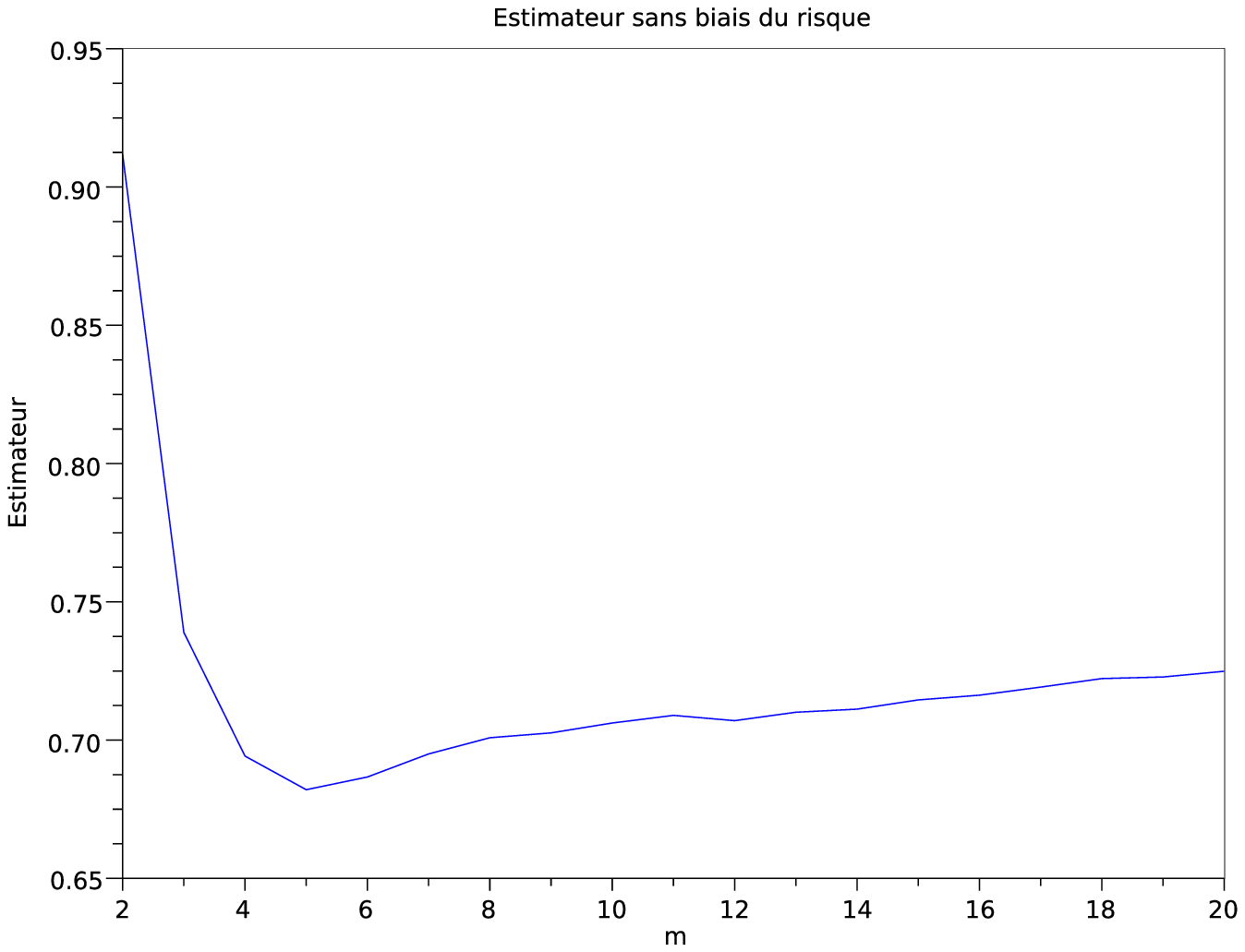} \\
\includegraphics[width=7 cm]  {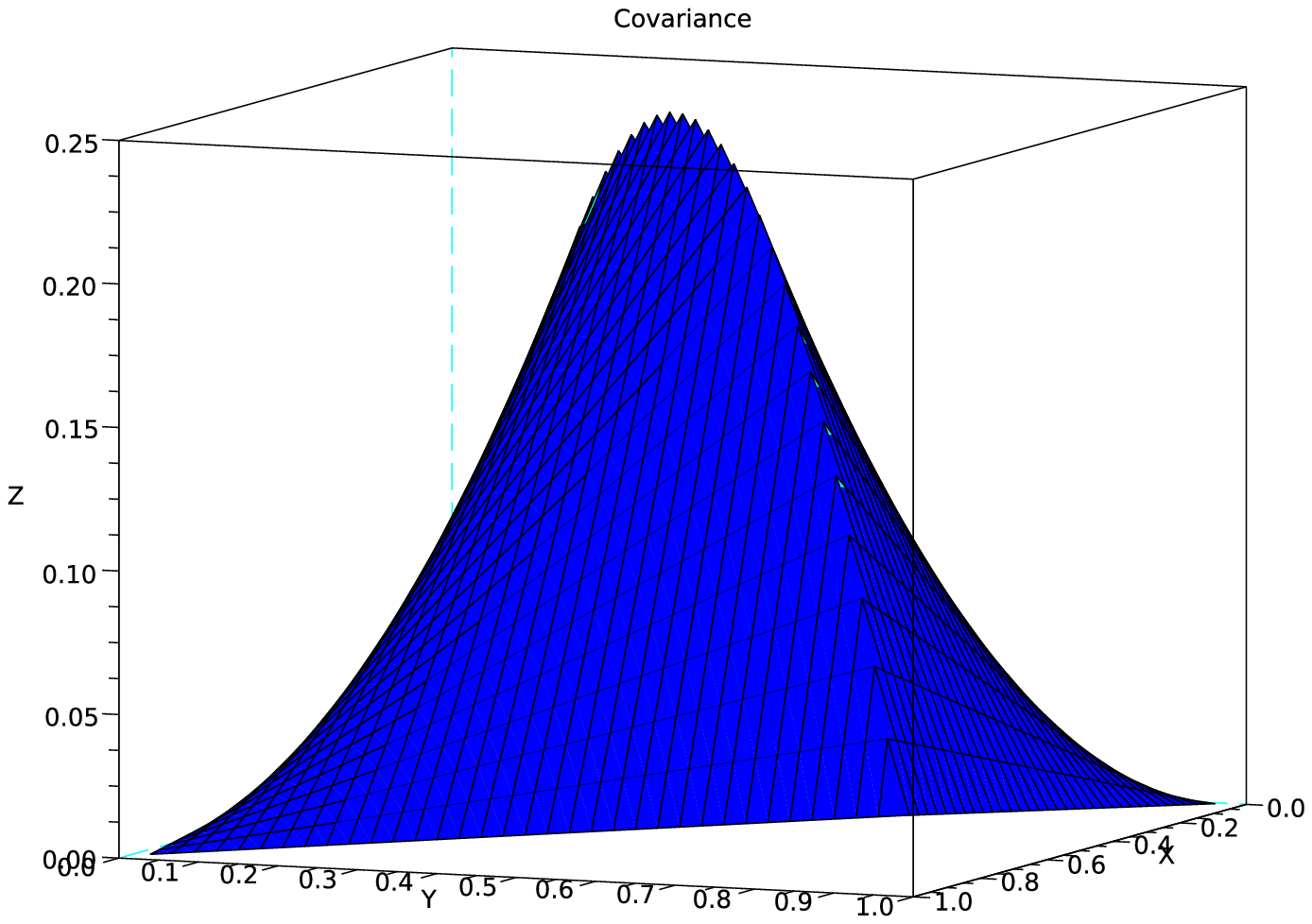}&
\includegraphics[width=7 cm]  {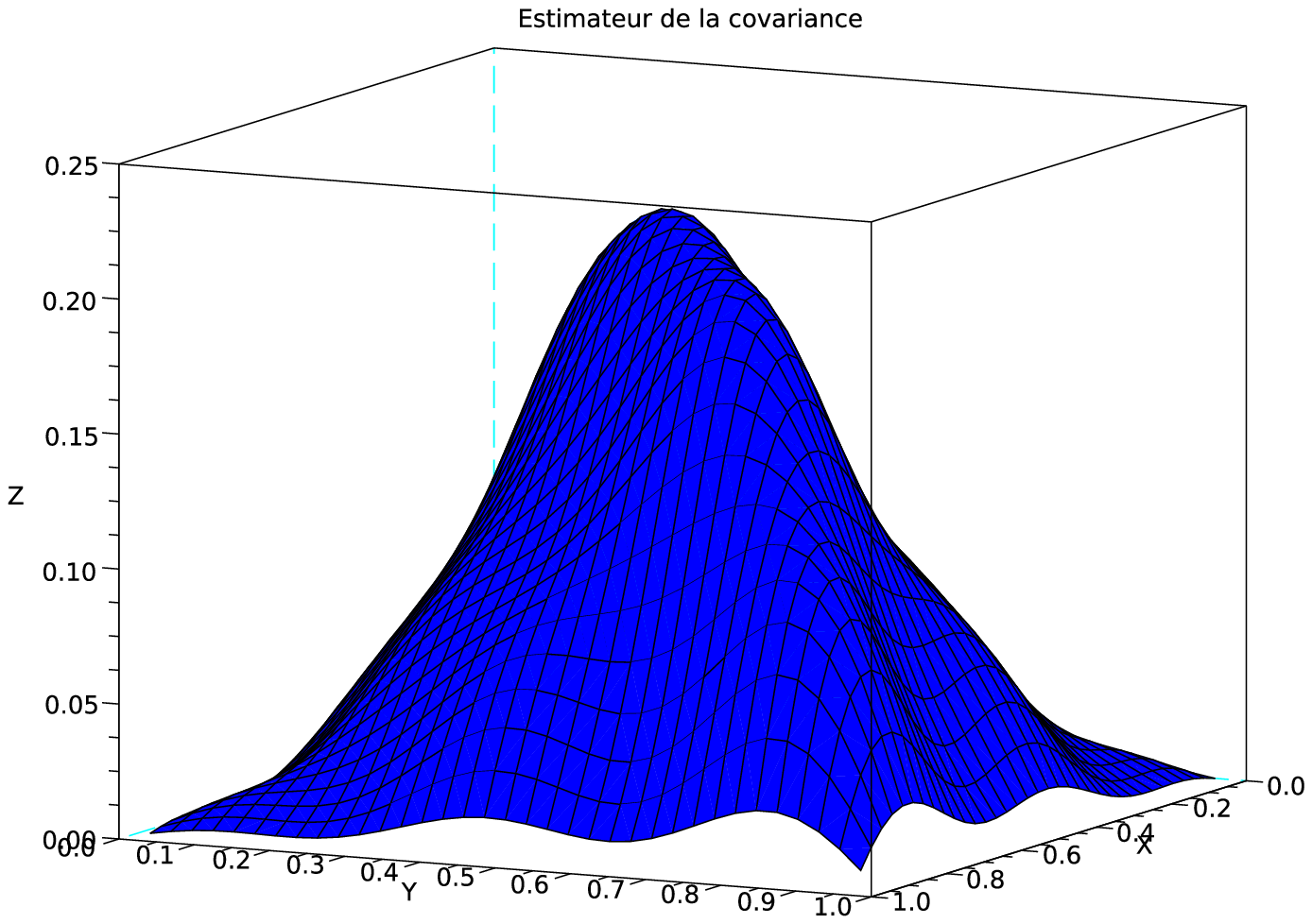}
\end{tabular}

Note that for the first and the last example,the size of the sample is not so large. However, for each of the simulated examples, the covariance estimator shows good performances. Indeed, the procedure introduced in this paper leads each time to the selection of  the best model, in the sense that  the  chosen model minimizes the risk.

\section{Appendix}

Recall that $\Sigma_m= \Pi_m \Sigma \Pi_m$, $\hat{\Sigma}_m= \Pi_m S \Pi_m$ and
$$\hat{\gamma}_m^2= \frac{1}{n-1}\sum_{i=1}^n \left\Vert\Pi_m x_ix_i^\top \Pi_m - \hat{\Sigma}_m \right\Vert^2$$

We start by proving the proposition \ref{eq:risk}.

\begin{proof}
Using the orthogonality, we have
$$\left\Vert \Sigma - \hat{\Sigma}_m \right\Vert^ 2 = \left\Vert \Sigma - \Sigma_m \right\Vert^ 2+\left\Vert \Sigma_m - \hat{\Sigma}_m \right\Vert^ 2$$
With the proposition \ref{rappel} we deduce
$$\left\Vert \Sigma_m - \hat{\Sigma}_m \right\Vert^ 2=\left\Vert vec(\Sigma_m - \hat{\Sigma}_m )\right\Vert^ 2=\left\Vert \left(\Pi_m^\top \otimes \Pi_m\right)vec\left( \Sigma - S\right) \right\Vert^ 2$$
Since $\Pi_m$ is a projection matrix,
$$\left\Vert \Sigma_m - \hat{\Sigma}_m \right\Vert^ 2= tr\left(\left(\Pi_m \otimes \Pi_m\right) vec\left( \Sigma - S\right)vec\left( \Sigma - S\right)^ \top\right)$$

Hence
$$\E \left[ \left\Vert \Sigma - \hat{\Sigma}_m \right\Vert^ 2 \right]=\left\Vert \Sigma - \Sigma_m \right\Vert^ 2 + \E\left[tr\left(\left(\Pi_m \otimes \Pi_m\right) vec\left( \Sigma - S\right)vec\left( \Sigma - S\right)^ \top\right)\right]$$
$$=\left\Vert \Sigma - \Sigma_m \right\Vert^ 2 + tr\left(\left(\Pi_m \otimes \Pi_m\right) \E\left[vec\left( \Sigma - S\right)vec\left( \Sigma - S\right)^ \top\right]\right)$$

$$=\left\Vert \Sigma - \Sigma_m \right\Vert^ 2 + \frac{tr\left(\left(\Pi_m \otimes \Pi_m\right) \E\left[vec\left( \Sigma - xx^\top\right)vec\left( \Sigma - xx^\top\right)^ \top\right]\right)}{n}$$
\end{proof}

\textbf{Proof of Proposition \ref{ure}.}

\begin{proof}

We start by the proof of the following lemma
\begin{lem}
$\hat{\gamma}_m^2$ is an unbiased estimator of $tr\left(  \left( \Pi_m \otimes \Pi_m \right)  \Phi\right)$.
\end{lem}

\begin{proof}
We deduce from the proposition \ref{rappel} and the fact that $\Pi_m$ is a projection matrix that:
$$\left( n-1\right)  \mathbb{E}\left[   \hat{\gamma}_m^2\right]  =\sum_{i=1}^n \mathbb{E}\left[    \left\Vert vec\left( \Pi_m x_ix_i^\top \Pi_m\right)   - vec\left( \hat{\Sigma}_m\right)   \right\Vert^2\right]  $$
$$=\sum_{i=1}^n \mathbb{E}\left[    \left\Vert\left( \Pi_m \otimes \Pi_m\right)   \left(  vec\left(  x_ix_i^\top \right)   - vec\left( S\right)  \right)   \right\Vert^2\right]  $$
$$=\sum_{i=1}^n \mathbb{E}\left[   tr\left(  \left( \Pi_m \otimes \Pi_m\right)   \left(  vec\left(  x_ix_i^\top \right)   - vec\left( S\right)  \right)  \left(  vec\left(  x_ix_i^\top \right)   - vec\left( S\right)  \right)  ^\top \left( \Pi_m \otimes \Pi_m\right)  ^\top\right)  \right]  $$

$$=\sum_{i=1}^n tr\left( \left( \Pi_m \otimes \Pi_m\right)   \mathbb{E}\left[    \left(  vec\left(  x_ix_i^\top \right)   - vec\left( S\right)  \right)  \left(  vec\left(  x_ix_i^\top \right)   - vec\left( S\right)  \right)  ^\top \right]   \right)    $$

But if $\left( v_i\right)  _{1 \leqslant i \leqslant n}$, are some i.i.d. vectors with covariance matrix $V$ and mean $\bar{v}= \frac{1}{n}\sum_{i=1}^n v_i$, we have
$$\mathbb{E}\left[   \left( v_i-\bar{v}\right)  \left( v_i-\bar{v}\right)  ^\top\right]  =\frac{1}{n^2}\sum_{j,k=1}^n \mathbb{E}\left[   \left( v_i-v_k\right)  \left( v_i-v_j\right)  ^\top\right]  $$
$$=\frac{1}{n^2}\sum_{\substack{j,k=1 \\ j,k \neq i}}^n \mathbb{E}\left[   \left( v_i-v_k\right)  \left( v_i-v_j\right)  ^\top\right]  $$
$$=\frac{1}{n^2}\left\lbrace \left( n-1\right)  \mathbb{E}\left[   \left( v_1-v_2\right)  \left( v_1-v_2\right)  ^\top\right]   + \left( n-2\right)  \left( n-1\right)  \mathbb{E}\left[   \left( v_1-v_2\right)  \left( v_1-v_3\right)  ^\top\right]   \right\rbrace $$
$$=\frac{1}{n^2}\left\lbrace \left( n-1\right)  2V +  \left( n-2\right)  \left( n-1\right)    V\right\rbrace$$
$$=\frac{1}{n^2}\left( \left( n-1\right)  nV\right)  $$
Hence
$$\mathbb{E}\left[   \left( v_i-\bar{v}\right)  \left( v_i-\bar{v}\right)  ^\top\right]  =\frac{1}{n}\left( \left( n-1\right)  V\right)  $$
this identity gives 
$$\left( n-1\right)  \mathbb{E}\left[   \hat{\gamma}_m^2\right]  =\sum_{i=1}^n tr\left( \left( \Pi_m \otimes \Pi_m\right)  \frac{1}{n}\left( \left( n-1\right)  \Phi\right)  \right)    $$
Finally

$$\mathbb{E}\left[   \hat{\gamma}_m^2\right]  =tr\left(  \left( \Pi_m \otimes \Pi_m\right)  \Phi\right)  $$

\end{proof}

Now, it remains to show that 
$$\mathbb{E}\left[    \left\Vert S- \hat{\Sigma}_m\right\Vert^ 2\right]   = \left\Vert \Sigma - \Pi_m \Sigma \Pi_m\right\Vert^2 - \frac{tr\left( \left(  \Pi_m \otimes \Pi_m\right)   \Phi\right)  }{n} +  \frac{tr\left( \Phi\right)  }{n}$$

We have that
$$\left\Vert S- \hat{\Sigma}_m\right\Vert^ 2 = \left\Vert S- \Sigma\right\Vert^ 2+2\left\langle  S- \Sigma, \Sigma - \hat{\Sigma}_m\right\rangle   + \left\Vert\Sigma- \hat{\Sigma}_m\right\Vert^ 2$$
And using the orthogonality we deduce that
$$\left\Vert S- \hat{\Sigma}_m\right\Vert^ 2 = \left\Vert S- \Sigma\right\Vert^ 2+2\left\langle  S- \Sigma, \Sigma - \hat{\Sigma}_m\right\rangle   + \left\Vert\Sigma- \Sigma_m\right\Vert^ 2 + \left\Vert\Sigma_m -\hat{\Sigma}_m\right\Vert^ 2$$
For the same reason :
$$\left\langle  S- \Sigma, \Sigma - \hat{\Sigma}_m\right\rangle  =\left\langle   S-\Sigma, \Sigma - \Sigma_m\right\rangle   +\left\langle   S-\Sigma, \Sigma_m - \hat{\Sigma}_m\right\rangle    $$
$$=\left\langle   S-\Sigma, \Sigma - \Sigma_m\right\rangle   - \left\Vert\Sigma_m - \hat{\Sigma}_m\right\Vert^ 2$$
And because the expectation of $S$ is equal to $\Sigma$ we obtain that
$$\mathbb{E}\left[   \left\Vert S- \hat{\Sigma}_m\right\Vert^ 2\right]  = \left\Vert\Sigma- \Sigma_m\right\Vert^ 2 + \mathbb{E}\left[   \left\Vert S- \Sigma\right\Vert^ 2\right]  -\mathbb{E}\left[   \left\Vert\Sigma_m -\hat{\Sigma}_m\right\Vert^ 2\right]  $$

First
$$\mathbb{E}\left[   \left\Vert S- \Sigma\right\Vert^ 2\right]  =\frac{1}{n^ 2} \mathbb{E}\left[   \sum_{i,j=1}^ n \left\langle   x_ix_i^\top - \Sigma, x_jx_j^\top - \Sigma\right\rangle  \right]  =\frac{1}{n}\mathbb{E}\left[   \left\Vert xx^\top-\Sigma\right\Vert^2\right]  $$
And with the properties of the Frobenius norm
$$\mathbb{E}\left[   \left\Vert xx^\top-\Sigma\right\Vert^2\right]  = \mathbb{E}\left[   \left\Vert vec\left( xx^\top-\Sigma\right)   \right\Vert^ 2\right]  $$
$$= tr\left(\mathbb{E}\left[   \left( vec\left( xx^\top\right)   - vec\left( \Sigma\right)  \right)  \left( vec\left( xx^\top\right)   - vec\left( \Sigma\right)  \right)  ^\top\right]  \right)$$
then we derive that
$$\mathbb{E}\left[   \left\Vert xx^\top-\Sigma\right\Vert^2\right]  =tr\left( \Phi\right)  $$
Thus
\begin{equation}
\label{eq0}
\mathbb{E}\left[   \left\Vert S- \Sigma\right\Vert^ 2\right]  =\frac{tr\left( \Phi\right)  }{n}
\end{equation}

Second
$$\mathbb{E}\left[   \left\Vert\Sigma_m -\hat{\Sigma}_m\right\Vert^ 2\right]  = \frac{1}{n^ 2} \mathbb{E}\left[   \sum_{i,j=1}^ n \left\langle  \Pi_m\left(  x_ix_i^\top - \Sigma\right)  \Pi_m, \Pi_m\left( x_jx_j^\top - \Sigma\right)  \Pi_m\right\rangle  \right]  $$
$$=\frac{1}{n} \mathbb{E}\left[    \left\langle  \Pi_m\left(  xx^\top - \Sigma\right)  \Pi_m, \Pi_m\left( xx^\top - \Sigma\right)  \Pi_m\right\rangle  \right]   =\frac{1}{n}\mathbb{E}\left[   \left\Vert\Pi_m\left(  xx^\top - \Sigma\right)  \Pi_m\right\Vert^ 2\right]   $$

And using the proposition \ref{rappel} and the specificity of $\Pi_m$, we obtain that

$$\mathbb{E}\left[   \left\Vert \Pi_m\left(  xx^\top - \Sigma\right)  \Pi_m\right\Vert^ 2\right]   = \mathbb{E}\left[   \left\Vert vec\left( \Pi_m\left(  xx^\top - \Sigma \right)  \Pi_m\right)\right\Vert^ 2\right]   $$
$$=\mathbb{E}\left[   \left\Vert\Pi_m\otimes \Pi_m\left( vec\left( xx^\top - \Sigma\right)  \right)  \right\Vert^2\right]  $$ 
$$=\mathbb{E}\left[    tr \left( \Pi_m\otimes \Pi_m\left( vec\left( xx^\top - \Sigma\right)  \left( vec\left( xx^\top - \Sigma\right)  \right)  ^\top\left( \Pi_m\otimes \Pi_m\right)  ^\top\right)     \right)\right]  $$

$$=\mathbb{E}\left[   tr\left( \left( \Pi_m\otimes \Pi_m\right) \left( vec\left( xx^\top - \Sigma\right) \right) \left( vec\left( xx^\top - \Sigma\right)\right)  ^\top \right) \right]  $$
Hence
$$\mathbb{E}\left[   \left\Vert\Pi_m\left(  xx^\top - \Sigma\right)  \Pi_m\right\Vert^ 2\right]   = tr\left( \left( \Pi_m\otimes \Pi_m\right)   \Phi\right)  $$
And we obtain
\begin{equation}
\mathbb{E}\left[   \left\Vert\Sigma_m -\hat{\Sigma}_m\right\Vert^ 2\right]  =\frac{tr\left( \left( \Pi_m\otimes \Pi_m \right)  \Phi\right)  }{n}
\end{equation}

Finally, we have
$$\mathbb{E}\left[    \left\Vert S- \hat{\Sigma}_m\right\Vert^ 2 \right]  =\left\Vert\Sigma - \Sigma_m\right\Vert^2 - \frac{tr\left( \left( \Pi_m \otimes \Pi_m\right)   \Phi\right)  }{n}+\frac{tr\left( \Phi\right)  }{n}$$


\end{proof}

\textbf{Proof of Proposition \ref{oracle}.}

\begin{proof}
As $\frac{\hat{\gamma}_m^2}{n}\geqslant 0$, we have: 
$$\mathbb{E}\left[   \left\Vert\hat{\Sigma}- \Sigma \right\Vert^2\right]   \leqslant \mathbb{E}\left[   \left\Vert\hat{\Sigma}- S \right\Vert^2+ 2 \frac{\hat{\gamma}_{\hat{m}}^2}{n}\right]   +2 \mathbb{E}\left[   \left\langle  \hat{\Sigma}- S,S- \Sigma \right\rangle  \right]   +\mathbb{E}\left[   \left\Vert S- \Sigma \right\Vert^2 \right]  $$
Let $m_0 \in \underset{ m\in \mathcal{M}}{argmin} \mathbb{E}\left[   \left\Vert \Sigma -\hat{\Sigma}_m \right\Vert^ 2\right]  $ an oracle.
By definition of $\hat{m}$, 
$$\left\Vert S- \hat{\Sigma}_{\hat{m}}\right\Vert^ 2 + 2 \frac{\hat{\gamma}_{\hat{m}}^2}{n} \leqslant \left\Vert S- \hat{\Sigma}_{m_0}\right\Vert^ 2 + 2 \frac{\hat{\gamma}_{m_0}^2}{n}$$
Then
$$\mathbb{E}\left[   \left\Vert\hat{\Sigma}- \Sigma \right\Vert^2\right]   \leqslant \mathbb{E}\left[   \left\Vert S- \hat{\Sigma}_{m_0}\right\Vert^ 2 + 2 \frac{\hat{\gamma}_{m_0}^2}{n}\right]   + \mathbb{E}\left[   \left\Vert S- \Sigma \right\Vert^2 \right]   + 2 \mathbb{E}\left[   \left\langle  \hat{\Sigma}- S,S- \Sigma \right\rangle  \right]  $$
we derive from the previous proposition and \eqref{eq0}
$$\mathbb{E}\left[   \left\Vert\hat{\Sigma}- \Sigma \right\Vert^2\right]   \leqslant \mathbb{E}\left[   \left\Vert \Sigma -\hat{\Sigma}_{m_0} \right\Vert^ 2\right]   + 2\frac{tr\left( \Phi\right)  }{n} + 2 \mathbb{E}\left[   \left\langle  \hat{\Sigma}- S,S- \Sigma \right\rangle  \right]  $$
Moreover by the Cauchy-Schwarz inequality we have that
$$\left\langle  \hat{\Sigma}- S,S- \Sigma \right\rangle  \leqslant \left\Vert\hat{\Sigma}- S \right\Vert \left\Vert S- \Sigma\right\Vert$$
And using again this inequality
$$\mathbb{E}\left[   \left\langle  \hat{\Sigma}- S,S- \Sigma \right\rangle  \right]  \leqslant \sqrt{\mathbb{E}\left[   \left\Vert\hat{\Sigma}- S\right\Vert^2\right]  }\sqrt{\mathbb{E}\left[   \left\Vert S- \Sigma\right\Vert^2\right]  }$$
$$\leqslant \sqrt{\mathbb{E}\left[   \left\Vert\hat{\Sigma}- S\right\Vert^2 + 2\frac{\hat{\gamma}_{\hat{m}}^2}{n}\right]  }\sqrt{\frac{tr\left( \Phi\right)  }{n}}$$
For the same reasons as before we obtain
$$\mathbb{E}\left[   \left\langle  \hat{\Sigma}- S,S- \Sigma \right\rangle  \right]  \leqslant \sqrt{ \mathbb{E}\left[   \left\Vert \Sigma- \hat{\Sigma}_{m_0}\right\Vert^ 2\right]+\frac{tr\left( \Phi\right)  }{n}   }\sqrt{\frac{tr\left( \Phi\right)  }{n}  }$$
$$\leqslant \frac{tr\left( \Phi\right)  }{n} + \sqrt{\mathbb{E}\left[   \left\Vert \Sigma - \hat{\Sigma}_{m_0}\right\Vert^ 2\right]  }\sqrt{\frac{tr\left( \Phi\right)  }{n}}$$
Thus
$$\mathbb{E}\left[   \left\Vert\hat{\Sigma}- \Sigma \right\Vert^2\right]   \leqslant \mathbb{E}\left[   \left\Vert \Sigma -\hat{\Sigma}_{m_0} \right\Vert^ 2\right]   +4\frac{tr\left( \Phi\right)  }{n} + 2\sqrt{\mathbb{E}\left[   \left\Vert S- \hat{\Sigma}_{m_0}\right\Vert^ 2\right]  }\sqrt{\frac{tr\left( \Phi\right)  }{n}}$$
With the following inequality which holds $\forall a$, $ b\in \mathbb{R}$ et $\forall A>0$
$$2ab \leqslant \frac{a^2}{A} + A b^2$$
We obtain for all $A>0$:
$$\mathbb{E}\left[   \left\Vert\hat{\Sigma}- \Sigma \right\Vert^2\right]   \leqslant \mathbb{E}\left[   \left\Vert \Sigma -\hat{\Sigma}_{m_0} \right\Vert^ 2\right]  \left( 1+A^{-1}\right)   + \frac{tr\left( \Phi\right)  }{n}\left( 4+A\right)  $$
The definition of $m_0$ gives the result.
\end{proof}

\bibliographystyle{alpha}
\bibliography{base}

\end{document}